\documentclass{article}

\usepackage{amsthm}
\usepackage{amsmath}
\usepackage{amssymb}
\usepackage{framed}
\usepackage{graphicx}
\theoremstyle{definition}
\newtheorem{theorem}{Theorem}
\newtheorem{lemma}[theorem]{Lemma}

\def\b0{{\bf 0}}
\def\b1{{\bf 1}}

\title{
All bipartite circulants are dispersable}
\author{Shannon Overbay, Samuel Joslin, Paul C. Kainen}
\date{}

\begin{document}
\maketitle

\abstract{
We show that a cyclic vertex order due to Yu, Shao and Li gives a dispersable book embedding for any bipartite circulant.\\

}

\bigskip
\bigskip


It has been conjectured \cite{alam-2021} that vertex transitive bipartite graphs can be laid out with vertices on the unit circle, with edges as chords, such that there is a Vizing type-1  edge-coloring where no crossings are monochromatic.  Such a drawing and coloring is a {\it dispersable book embedding} \cite{bk79}.

We proposed this in 1979, omitting vertex transitivity \cite{bk79}, but in 2018, Alam et al \cite{alam-2018} found two counterexamples 
and later Alam et al. \cite{alam-2021} showed the existence of an infinite family of regular bipartite graphs, of fixed degree, which require an arbitrarily large number of edge colors to avoid monochrome edge-crossings.

Based on their examples, Alam et al \cite{alam-2021} added the qualification of vertex transitivity, 
and we further conjectured in \cite{kjo} that vertex transitive {\it non}-bipartite graphs have a type-{\it 2} edge coloring (i.e., are {\it nearly dispersable}). See \cite{jko,pck-2011,shao-geng,shao-zeling, shao-liu-li} which support both conjectures.

In \cite[Theorem 4.1]{ysl}, Yu, Shao and Li found what we call the
YSL-order; see Fig. 1.
They used it only for bipartite degree 3 and 4 circulants. In contrast, we show that 
the following holds and give a short proof.
\begin{theorem}
Any bipartite circulant with the  YSL-order
is dispersable. 
\end{theorem}

The theorem is additional evidence for the conjecture of \cite{alam-2021}. Circulants have been one of the more widely used graph-types and so the theorem might have interesting applications.  See, e.g., \cite{harary-circ, circ2022, hwang, jok}. 

In fact, the YSL-order also gives dispersable embeddings for the Franklin and Heawood graphs but not for the Desargues graph.

\bigskip

We now briefly sketch some definitions for the reader's convenience.

A graph $G=(V,E)$ is a {\bf circulant} if $V = \{1,\ldots, n\}$, $n \geq 3$, and there exists 
$S \subseteq \{1, 2, \ldots, \lfloor n/2 \rfloor\}$ such that
$E = \{ij : j=i+s, s \in S\}$, where addition is mod $n$.  Such a circulant is denoted $C(n,S)$ and the elements in $S$ are the {\bf jump-lengths}. 
The ${\bf C}_{\bf n}$-{\bf distance} between two vertices $u, w$ of a circulant $C(n,S)$ is the graph-theoretic distance in $C_n$; for instance, the $C_{16}$-distance from 2 to 13 is 5.

A graph $G$ is {\bf dispersable} \cite{bk79} if it has an outerplane drawing (crossings allowed) and an edge-coloring with $\Delta(G)$ colors ($\Delta =$ maximum degree) such that two edges of the same color neither cross nor share an endpoint.  The color-classes (or {\bf pages}) are {\it matchings} for this book embedding.





\begin{figure}[ht!]
\centering
\includegraphics[width=90mm]{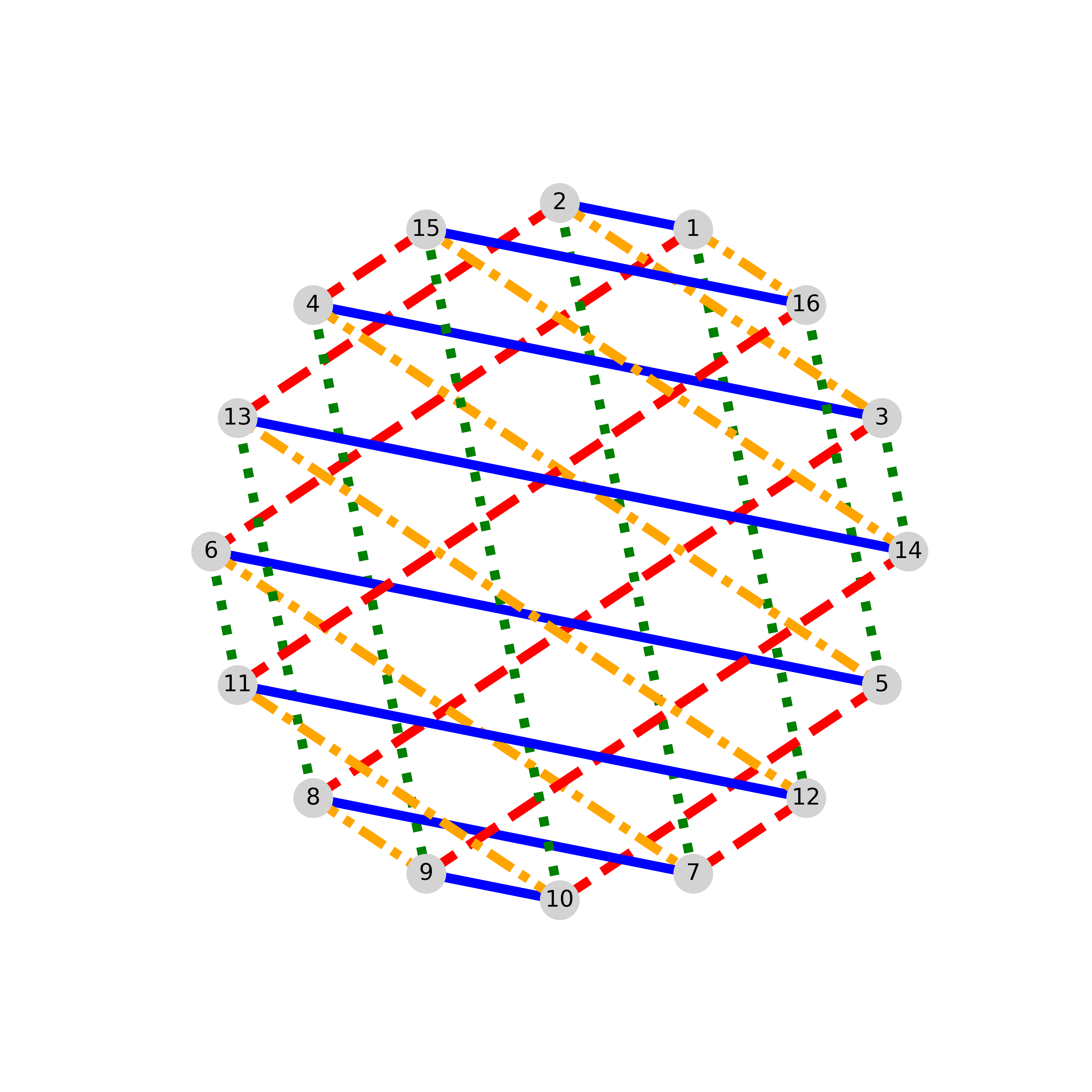}
\caption{ YSL-order for $C(16,\{1,3,5,7\})$. Only 1 and 5 are shown.
\label{fig-12reg}}
\end{figure}

In a {\bf YSL-order}, the natural clockwise cyclic order $1, 2, \ldots, 2k$ is permuted so that the odd vertices remain, in clockwise order, at the odd positions, while  the even vertices are placed at the even-indexed positions but in {\it counterclockwise} order. Our argument below shows that it doesn't matter where 2 is placed.  In \cite{ysl}, Yu et al put 2 immediately counterclockwise of 1, so in clockwise direction, their order is
$$ \dots, 2k-3, 4, 2k-1, 2, 1, 2k, 3, 2k-2, 5, \ldots$$



Using characterizations of bipartite circulants by Heuberger \cite{heuberger} and of connected circulants by Boesch and Tindell \cite{boesch-tindell}, and the fact that a disjoint union of identical subgraphs is dispersable if the subgraph is dispersable, one sees that to prove Theorem 1, 
it suffices to show the following.

\begin{theorem}
Let $n=2k$, $k \geq 2$, and let $\mu(k)$ be the largest odd number not exceeding $k$. Then $C = C(n, \{1, 3, \ldots, \mu(k)\})$ is dispersable under a YSL order, and the edges in each page correspond to a single jump-length.
\end{theorem}
\begin{proof}
Each page consists of a maximal parallel family of edges which we enumerate $1, \ldots, k$ in the order they intersect some orthogonal line.  We show that (i) the $C_n$-distance between the endpoints of edge 1 and edge 2 are equal and (ii) the same holds for edge $j$ and edge $j+2$, $1 \leq j \leq n-2$.  

For (i), suppose that $ab$ and $a'b'$ are edges 1 and 2, with $(a', a, b, b')$ occurring consecutively clockwise.  Then $a' = b \pm 2$ and $b' = a \pm 2$, where the signs agree because vertices alternate odd/even and the ordering is opposite for the two parities.
Hence, the $C_n$ distance from $a$ to $b$ is the same as the $C_n$ distance from $b'$ to $a'$.  But distance is symmetric.

For (ii), suppose that $ab$ and $a'b'$ are edges $j$ and $j+2$, with 
$$(a', x, a,\ldots, b, y, b')$$ 
occurring consecutively clockwise but $a$ and $b$ can be nonconsecutive. Then $a' = a \pm 2$ and $b' = b \pm 2$, where again both signs must be equal.
\end{proof}

For the reader's convenience, here are
the cited theorems.

\begin{theorem}[Heuberger \cite{heuberger}]
Let $C := C(n,\{a_1,\ldots,a_m\})$. Then
$C$ is bipartite if and only if there exists 
$\ell \in \mathbb{N}$ such that $2^\ell | a_1, \ldots, a_m$, $2^{\ell+1} | n$, but $2^{\ell + 1}$ does {\it not} divide any of the $a_j$.
\end{theorem}

For $t$ a positive integer, let $t H$ denote $t$ disjoint copies of graph $H$.
\begin{theorem}[Boesch \& Tindell \cite{boesch-tindell}]
Let $C := C(n,\{a_1,\ldots,a_m\})$ be a circulant.   Then $C = r \, C(n/r, \{a_1/r, \ldots, a_m/r\})$,
$r := \gcd(n, a_1, \dots, a_m)$.
\end{theorem}
If $C$ is bipartite, then $r = 2^\ell u$ for $u$ odd, with $\ell$ as in Theorem 3.

Each maximal family of parallel edges, under the YSL ordering, has constant jump-length which could represent delay, allowing delay to depend on direction.  In contrast, for Overbay's ordering \cite[p 82]{so-thesis}, for a family of parallel edges, jump-lengths vary palindromically and unimodally $1,3,\ldots, \mu(k), \ldots, 3,1$, with $\mu(k)$ repeated if and only if $k$ is even, as one traverses them with respect to a perpendicular line. See Fig. 2.

\begin{figure}[ht!]
\centering
\includegraphics[width=60mm]{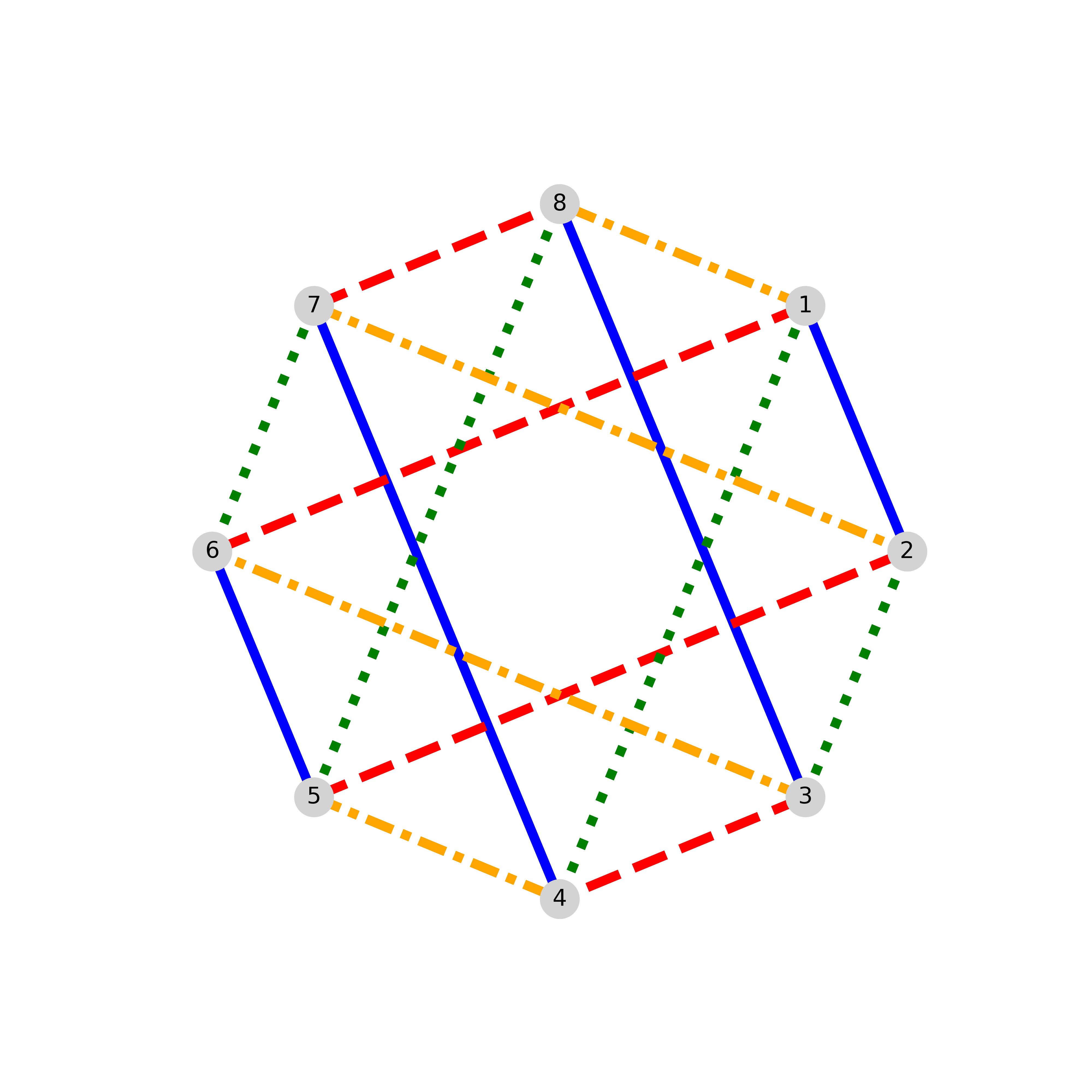}
\includegraphics[width=60mm]{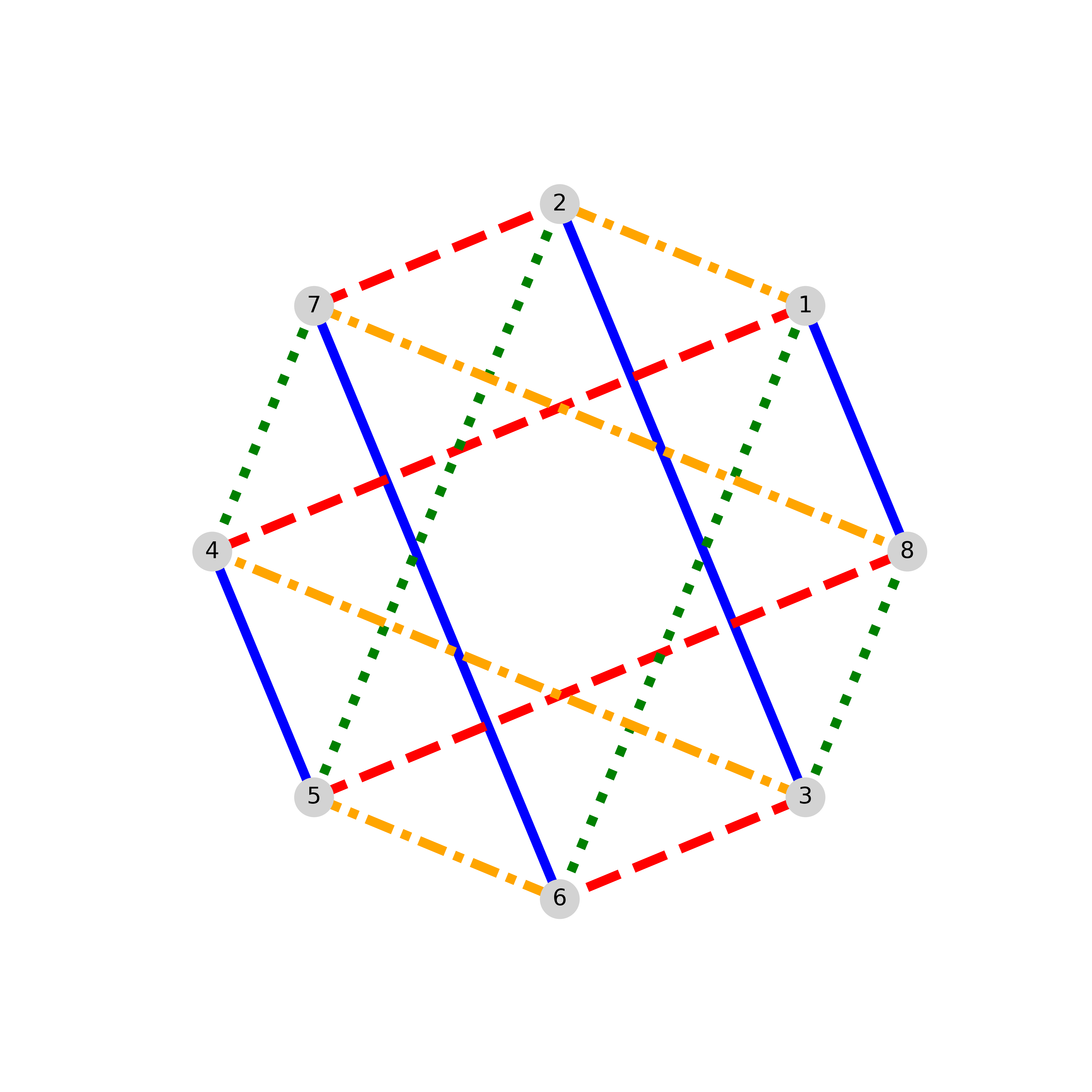}

\caption{ Overbay and YSL orders for $C(8,\{1,3\})$.\label{fig-12reg}}
\end{figure}


In applications, the bipartition could be input and output vertices.


\begin{thebibliography}{99}

\bibitem{alam-2018}
J. Md. Alam, M. A. Bekos, M. Gronemann, M. Kaufmann \& S. Pupyrev, On dispersable book embeddings, {\it Graph-theoretic Concepts in Computer Science} (44th International Workshop, WG 2018, Cottbus, Germany), A. Brandest\"{a}dt, E. K\"{o}hler, \& K. Meer, Eds., LNCS {\bf 11159} (2018) 1--14, Springer, Cham, Switzerland.

\bibitem{alam-2021}
J. Md. Alam, M. A. Bekos, V. Dujmovi\v{c}, M. Gronemann, M. Kaufmann \& S. Pupyrev, On dispersable book embeddings, {\it Theor. Computer Sci.} {\bf 861} (2021) 1--22.

\bibitem{bk79}
F. R. Bernhart \& P. C. Kainen, The book thickness of a graph, {\it J. Comb. Th.} {\bf B 27} (1979) 320--331.

\bibitem{boesch-tindell}
F. Boesch  \& R. Tindell, Circulants and their connectivities, {\it J. Graph Theory} {\bf 8} (1984) 487--499.

\bibitem{harary-circ}
F. Harary, The maximum connectivity of a graph, {\it Proc. National Acad. of Sci.} {\bf 48} (1962) 1142--1146.

\bibitem{heuberger}
C. Heuberger, On planarity and colorability of circulant graphs, {\it Discr. Math.} {\bf 268} (2003) 153--169.


\bibitem{circ2022}
X. Huang, A. F. Ramos, \& Y. Deng, Optimal circulant graphs as low-latency
network topologies, {\it J Supercomput} {\bf 78}(2022) 13491--13510 (2022). https://doi.org/10.1007/s11227-022-04396-5


\bibitem{hwang}
F.K. Hwang, A survey on multi-loop networks, {\it Theor. Comp. Science} {\bf 299} (2003) 107--121.

\bibitem{jko}
S. Joslin, P. C. Kainen \& S. Overbay,  Dispersability of some cycle products, {\it Missouri J. Math. Sci.} {\bf 33}(2) (2021) 206--213

\bibitem{jok}
S. Joslin, S. Overbay \& P. C. Kainen, Information as order (?).

\bibitem{pck-2011}
P. C. Kainen. Complexity of products of even cycles, {\it Bull. Inst. Comb. and Its Appl.} {\bf 62} (2011) 95--102.

\bibitem{kjo}
P. C. Kainen, S. Joslin \& S. Overbay, On dispersability of some circulants, {\it J Graph Algor. \& Appl.}, to appear.


\bibitem{so-thesis}
S. Overbay, Generalized book embeddings, {\it Ph.D. Dissertation}, Fort Collins: Colorado State University, 1998.

\bibitem{shao-geng}
Z. Shao, H. Geng, Z. Li, Matching book thickness of generalized Petersen graphs, {\it Elec. J. Graph Theory and Appl.}, {\bf 10}(1) (2022) 173--180.

\bibitem{shao-liu-li}
Z. Shao, Y. Liu, Z. Li. Matching book embedding of the Cartesian product of a complete graph and a cycle, {\it Ars Combinatoria} {\bf 153} (2020) 89--97.

\bibitem{shao-zeling}
Z. Shao, X. Yu, \& Z. Li, On the dispersability of toroidal grids, {\it Appl. Math. \& Computation} {\bf 453} (2023) 128087.

  
  \bibitem{ysl}
 X. Yu, Z. Shao \& Z. Li, On the classification and dispersability of circulant graphs with two jump lengths, {\it arXiv}:2310.06612 10 Oct 2023.

\end{thebibliography}
\end{document}